\documentclass[11pt, a4article]{article}
\usepackage{graphicx}
\usepackage{amscd}
\usepackage{amsmath}
\usepackage{amsfonts}
\usepackage{amssymb}
\usepackage{amsthm}
\usepackage{setspace}
\usepackage{enumerate}
\usepackage{color}
\usepackage{url}
\usepackage{hyperref}
\usepackage{bm}
\usepackage{xy}
\usepackage{stmaryrd}
\usepackage{a4wide}
\usepackage{fancyhdr}
\usepackage{leftidx}

\theoremstyle{plain}

\newtheorem{theorem}{Theorem}[section]

\newtheorem{corollary}[theorem]{Corollary}

\newtheorem{lemma}[theorem]{Lemma}

\newtheorem{proposition}[theorem]{Proposition}

\newtheorem{definition}[theorem]{Definition}

\newtheorem{assumption}[theorem]{Assumption}
\theoremstyle{remark}
\newtheorem{remark}[theorem]{Remark}

\numberwithin{equation}{section}

\newcommand{\ind}{1\!\kern-1pt \mathrm{I}}

\usepackage{authblk}
\providecommand{\keywords}[1]{\textbf{{Keywords.}} #1}
\providecommand{\subjclass}[1]{\textbf{{MSC2010.}} #1}
\bigskip

\title{Reflected stochastic differential equations driven by $G$-Brownian motion in non-convex domains}
\author{Yiqing Lin\thanks{Email: yiqing.lin@polytechnique.edu}\\ Centre de math\'ematiques appliqu\'ees\\ \'Ecole Polytechnique\\ 91128 Palaiseau Cedex -  France
\and Abdoulaye Soumana Hima \\Institut de recherche math\'ematiques de Rennes
\\Universit\'e de Rennes 1\\
35042 Rennes Cedex - France \\and\\
D\'epartement de math\'ematiques \\
Universit\'e de Maradi\\
BP 465 Maradi - Niger}
\date{\today}
\begin{document}
\maketitle
\begin{abstract}
In this paper, we first review the penalization method for solving deterministic Skorokhod problems in non-convex domains and establish estimates for problems with $\alpha$-H\"older continuous functions. With the help of these results obtained previously for deterministic problems, we pathwisely define the reflected $G$-Brownian motion and prove its existence and uniqueness in a Banach space.  Finally, multi-dimensional reflected stochastic differential equations driven by $G$-Brownian motion are investigated via a fixed-point argument. 
\end{abstract}


\noindent \keywords{H\"older continuity, Skorokhod problem, $G$-Brownian motion, Stochastic differential equations, non-convex reflecting boundaries}\\

\noindent \subjclass{60H10}

\maketitle

\section{Introduction}
\bigskip
This paper considers  multidimensional reflected stochastic differential equations in a sublinear expectation space. These equations are driven by a new type of Brownian motion associated with a sublinear expectation, which are introduced by Peng \cite{Peng3, Peng5, Peng4} during the past decade. In the seminal works of Peng, he has established a framework of nonlinear It\^o's calculus and related stochastic analysis which does not rely on a single probability measure. This theory provides basic tools to discuss problems in finance with Knigntian uncertainty and in robust statistics. 
Moreover, the nonlinear Feynman-Kac formula obtained in \cite{Peng4, HJPS14b} provides an  probabilistic representation of fully nonlinear parabolic PDEs via forward-backward systems. \\
%

In \cite{Peng4}, the so-called $G$-Brownian motion is defined as a continuous  process with stationary independent increments. Under the associated $G$-expectation, these increments are subject to $G$-normal distribution with volatility uncertainty between two bounds. According to Denis et al. \cite{DHP}, the $G$-expectation can be regarded as an upper expectation based on a collection of non-dominated martingale measures $\mathcal{P}_G$. Furthermore, a Choquet capacity associated with such a collection  can be defined. This leads to a notion of equivalence between two random variables in the sublinear expectation space --- ``quasi-sure'' (q.s.). Instead of ``almost-surely'' in the classical probability theory, we say a property holds quasi-surely if it holds outside a null set for the referred capacity.  In the present paper, we shall examine the following equation in the quasi-sure sense, 
\begin{equation}
\left\{
                \begin{aligned}
      &     X_{t}=x_{0}+\int_{0}^{t}f\left(
s,X_{s}\right) ds+\int_{0}^{t}h\left( s,X_{s}\right) d\langle B, B\rangle_{s}+\int_{0}^{t}g\left( s,X_{s}\right) dB_{s}+K_{t},~~~0\leq t\leq T;\\
     &             K_{t} =\int_{0}^{t}\mathbf{n}_s d\left\vert K\right\vert _{s};~~~
\left\vert K\right\vert _{t}=\int_{0}^{t}\mathbf{1}%
_{\left\{ X_{s} \in \partial D\right\} }d\left\vert
K\right\vert _{s},~~~~\text{q.s.}
                \end{aligned}
              \right.\label{a}
\end{equation}
where $B$ is a $d$-dimensional $G$-Brownian motion; $\langle B, B\rangle$ is the covariance matrix of $B$; $X$ is a  process reflecting on the boundaries of domain $\overline{D}$ and 
$K$ is a bounded variations process with variation $|K|$ increasing only when $X\in \partial\overline{D}$. \\

In the classical framework, reflected stochastic differential equations driven by Brownian motion have been extensively studied by many authors. Among them, Skorokhod \cite{Sko1961, Sko1962} is the first who introduced diffusion processes with
reflecting boundaries in the 1960s. Later on, reflecting diffusions in a half-space have been investigated by Watanabe \cite{watanabe}, El Karoui \cite{Elk1975}, Yamada \cite{Yamada}, El Karoui and Chleyat-Maurel \cite{Elk1978}, El Karoui et al. \cite{Elk1980}, etc. The study of multi-dimensional stochastic differential equations on a general domain dates back to Stroock and Varadhan \cite{stroock}, in which the existence and uniqueness of weak solutions  have been proved when the domain is smooth. Afterwards, solutions of such equations has been built on a convex domain by a direct method in Tanaka \cite{Tanaka}, whereas  
Menaldi \cite{menaldi} and Lions et al. \cite{lions} have adopted a penalization method to construct them. Concerning the reflecting problem with a non-convex but ``admissible'' domain, Lions and Sznitman \cite{LionsetSznitman} have first solved the deterministic Skorokhod  problem and have applied this result to construct pathwisely an iteration sequence in order to approximate  the corresponding reflecting diffusion. The results of \cite{LionsetSznitman} has been later improved in  Saisho \cite{Saisho1} and in Saisho and
Tanaka \cite{SaishoetTanaka} by removing the admissibility condition on the domain. \\

In the context of sublinear expectation, we shall discuss the above mentioned equation (\ref{a}), in which a newly defined It\^o stochastic integral with respect to $G$-Brownian motion $\int g dB$ appears in the dynamic.  As its counterpart in the classical theory, adapted to Peng's method, this integral is first defined for simple processes and could subsequently be extended to $M^p_G$ due to the $G$-BDG type estimate, where $M^p_G$ is a normed space of processes with slightly additional regularity (see also \cite{hwz}). Thanks to this extension of the It\^o type integral, the notion stochastic differential equations driven by $G$-Brownian motion is brought to this nonlinear stochastic analysis framework. Under the Lipschitz condtions, forward equations are studied by Peng \cite{Peng4} and Gao \cite{Gao}; backward equation are examined in \cite{HJPS14a}. Moreover, scalar $G$-diffusion processes with reflection has been considered by Lin in \cite{Lin13} and the multidimensional problem are solved in Lin \cite{LIN} by the penalization method similar to \cite{menaldi}.\\

The main objective of this paper is to generalize the results of \cite{Lin13, LIN} to the multidimensional case when the reflecting boundary is not necessarily convex. We adopt the same assumptions as \cite{Saisho1, SaishoetTanaka} on the domain and, as the first step, we restrict ourself to the deterministic Skorokhod problems concerning $\alpha$-H\"older continuous paths. Precisely, we revise the estimate for the penalization sequence obtained in \cite{SaishoetTanaka} by introducing the $\alpha$-H\"older coefficient. Since the $G$-Brownian motion is supported on $\mathcal{C}^{0, \alpha}$ ($\alpha<1/2$), this estimates are accordingly applied to prove that the multidimensional $G$-Brownian motion in this non-convex domain can be approximated in the Banach space $M^p_G$ by a sequence of solutions of Lipschitz equations studied in \cite{Gao}.  Similar arguments also apply for reflected $G$-diffusion with bounded generators. Finally, we pathwisely construct an iteration sequence by using the deterministic result and conduct a fixed-point argument as \cite{LionsetSznitman} to prove the wellposedness of (\ref{a}) under the bounded and Lipschitz assumption on coefficients. We remark that the equation (\ref{a}) under consideration can also be examined in a weaker ``quasi-sure'' sense, which means it holds $\mathbb{P}$-almost surely for all $\mathbb{P}\in \mathcal{P}_G$. This notion is adopted by Soner et al. for establishing a similar nonlinear stochastic analysis framework -- second order backward stochastic differential equations (cf. \cite{soner}). We could proceed almost the same procedures of the present paper under each $\mathbb{P}\in \mathcal{P}_G$ and obtain a solution $(X^\mathbb{P}, K^\mathbb{P})$ which could be later aggregated to $(X, K)$ by Nutz \cite{nutz}. \\

The paper is organized as follows. In Section 2, we introduce
preliminaries in the framework of $G$-expectation which are necessary for the remainder of this paper. In addition, we revisit the deterministic Skorohod problems in a non-convex domain. In Section 3, we present our main results and Section 4 is devoted to prove the main results.
\bigskip

\section{Preliminaires}
In this section, we shall briefly introduce the $G$-expectation framework. Moreover, we shall discuss the deterministic Skorohod problem in non-convex domains and the sufficient conditions for its solvability.
\subsection{$G$-Brownian motion and $G$-expectation}
Adapting to Peng's framework, we first recall useful notations and results on the $G$-expectation and related $G$-It\^o type stochastic calculus.
 The reader interested
in the more detailed description on this topic is referred to
Denis et al. \cite{DHP}, Gao \cite{Gao}, Li and Peng \cite{LIetPeng} and Peng \cite{Peng4}.\\

Let $\Omega$ be the space of all $\mathbb{R}^{d}$-valued
continuous paths with $\omega_0 =0$, noted by $\mathcal{C}_0([0, \infty); \mathbb{R}^d),$
equipped
with the distance
$$
\rho(\omega^1, \omega^2):=\sum^\infty_{N=1} 2^{-N} ((\max_
{t\in[0,N]} | \omega^1_t-\omega^2_t|)  \wedge 1),
$$
and $(B_t)_{t\geq 0}$ be the canonical process, i.e., $B_t(\omega):=\omega_t$. For each $t\in \lbrack0,\infty)$, we list the following notations:

\begin{itemize}
\item[$\bullet$] $\Omega_{t}:=\{ \omega_{\cdot \wedge t}:\omega \in \Omega \}$; $\mathcal{F}_{t}:=\mathcal{B}(\Omega_{t})$;
\item[$\bullet$] $L^{0}(\Omega)$: the space of all $\mathcal{B}(\Omega)$-measurable real functions;
\item[$\bullet$] $L_{ip}(\Omega_t):=\{\varphi(B_{t_1}, \cdots, B_{t_n}):n\geq
1,\ 0\leq t_1< \cdots< t_n\leq t,\ \varphi \in C_{b,lip}
(\mathbb{R}^{d\times n})\}$;
\item[$\bullet$] $L_{ip}(\Omega):=\{\varphi(B_{t_1}, \cdots, B_{t_n}):n\geq
1,\ 0\leq t_1< \cdots< t_n\leq \infty,\ \varphi \in C_{b,lip}
(\mathbb{R}^{d\times n})\}$;
\end{itemize}
where $C_{b,lip}
(\mathbb{R}^{d\times n})$ is the collection of all bounded Lipschitz functions on $\mathbb{R}^{d\times n}$.\\

We fix a sublinear continuous and monotone function $G$: $\mathbb{S}^d\rightarrow \mathbb{R}$. For some bounded and closed subset $\Gamma\subset \mathbb{R}^d$, this function can be represented by
$$
G(A_0)=\sup_{Q\in \Gamma}\left\{\frac{1}{2}{\rm tr}\left[A_0Q\ltrans{Q}\right]\right\}, \ \ \ \ {\rm for}\ A_0\in \mathbb{S}^d.
$$
The related $G$-expectation on $(\Omega, L_{ip}(\Omega))$ can be constructed in the following way: for each $\xi\in L_{ip}(\Omega)$ of the form
$$\xi=\varphi(B_{t_1},B_{t_2}-B_{t_1},\cdots,B_{t_n}-B_{t_{n-1}}), \ \ 0\leq t_1<t_2<\cdots<t_n,$$
define
$$\mathbb{E}_G[\xi]:=u_1(0,0),$$
where $u_1(0,0)\in\mathbb{R}$ is obtained by recurrence: for $k=n,\cdots,1$, define $u_k:=u_k(t,x;x_1,\cdots,x_{k-1})$, which is a function in $(t,x)$ parameterized by $(x_1,\cdots,x_{k-1})\in\mathbb{R}^{d\times(k-1)}$, by the solution of the following $G$-heat equation defined on $[t_{k-1},t_k)\times\mathbb{R}^d$:
 $$\frac{\partial u_k}{\partial t}-G(D^2 u_k)=0,$$
 with the terminal condition
 $$u_k(t_k,x;x_1,\cdots,x_{k-1})=u_{{k+1}}(t_{k},x;x_1,\cdots,x_{k-1},x).$$
In particular, $u_n(t_n,x;x_1,\cdots,x_{n-1}):=\varphi(x_1,\cdots,x_{n-1},x)$. We say that the canonical process $(B_t)_{t\geq 0}$ is a $G$-Brownian motion under this sublinear expectation $\mathbb{E}_G[\cdot]$, which is with stationary, independent and $G$-Gaussian distributed increments  (see Definition 1.4 and 1.8 in Chap. I of \cite{Peng4} for the definition of $G$-Gaussian distribution and see Definition 3.10 in  Chap. I of \cite{Peng4} for the definition of independence under sublinear expectation).\\


We denote by $L_G^p(\Omega)$ (resp. $L_G^p(\Omega_T)$ the completion of $Lip(\Omega)$ (resp. $Lip(\Omega_T)$) with respect to the norm $||\cdot||_p:=\mathbb{E}_G[|\cdot|^p]^{\frac{1}{p}}$, for $p\geq 1$.
We can extend the domain of $G$-expectation $\mathbb{E}_G[\cdot]$ from $L_{ip}(\Omega)$ to $L^{0}(\Omega)$ by the procedure introduced in \cite{DHP},
i.e., constructing an upper expectation $\mathbb{E}[\cdot]$:
\begin{equation*}
\mathbb{E}[X]:=\sup_{\mathbb{P}\in\mathcal{P}_G}{\mathbf E}^\mathbb{P}[X],\ X\in L^0(\Omega),
\end{equation*}
where $\mathcal{P}_G$ is a weakly compact family of martingale measures on $(\Omega,\mathcal{B}(\Omega))$. This upper expectation
coincides with the $G$-expectation $\mathbb{E}_G[\cdot]$ on $L_{ip}(\Omega)$ and thus, on its completion $L^1_G(\Omega)$. Naturally, the Choquet capacity related to the upper expectation can be defined by
$$
\bar{C}(A):=\sup_{\mathbb{P}\in\mathcal{P}_G}\mathbb{P}(A),\ A\in\mathcal{B}(\Omega),
$$
and the notation of ``quasi-surely'' (q.s.) can be introduced as follows:
\begin{definition}[Quasi-sure]
A set $A\in\mathcal{B}(\Omega)$ is called polar if $\bar{C}(A)=0.$ A property
is said to hold quasi-surely if it holds outside a polar
set.
\end{definition}

The following Markov's inequality holds in the context of the upper expectation and the related Choquet capacity (Lemma 13 in \cite{DHP}).
\begin{lemma}[Markov's inequality]\label{MI}
Let $X\in L^0(\Omega)$ satisfying $\mathbb{E}[|X|^p]<\infty$, for $p>0$. Then, for each $a>0$,
$$
\bar{C}(\{|X|>a\})\leq \frac{\mathbb{E}[|X|^p]}{a^p}.
$$
\end{lemma}

We also have a generalized Fatou's lemma (cf. e.g. Lemma 2.11 in Bai and Lin \cite{Bai-Lin}) in the $G$-framework.
\begin{lemma}[Fatou's lemma]\label{fatulem}
Assume that $\{X^n\}_{n\in\mathbb{N}}$ is a sequence in $L^0(\Omega)$ and that for a $Y\in L^0(\Omega)$ satisfying $\mathbb{E}[|Y|]<\infty$ and for all $n\in\mathbb{N}$, $X^n \geq Y$, q.s., then
$$
\mathbb{E}[\liminf_{n\rightarrow\infty}X^n]\leq\liminf_{n\rightarrow\infty}\mathbb{E}[X^n].\\
$$
\end{lemma}

In \cite{Peng4}, Peng introduce the It\^o type stochastic integral with respect to the $G$-Brownian motion by first considering the simple process space:
\begin{align*}
M^0([0,T]; \mathbb{R})  &  =\{ \eta:\eta_{t}(\omega)=\sum_{i=0}^{n-1}\xi_{i}%
(\omega)\mathbf{1}_{[t_{i},t_{i+1})}(t),\\
&\  {\rm where}\ n\in \mathbb{N}^*,\ 0=t_{0}<\cdots<t_{n}=T,\ \xi_{i}\in Lip%
(\Omega_{t_{i}}),\ i=0,\cdots,n-1\}.
\end{align*}
\begin{definition}
For $p\geq 1$, we denote by $M_G^p([0,T];\mathbb{R})$ the completion of $M^0([0,T]; \mathbb{R})$ under the following norm:
\begin{equation*}
\|\eta\|_p:=\bigg(\mathbb{E}\bigg[\frac{1}{T}\int_0^T|\eta_t|^p dt\bigg]\bigg)^{1/p}.
\end{equation*}
\end{definition}

Here below is the definition of the $G$-It\^o type sintegral. In the sequel, $B^{\bf a}$ denotes the inner product of ${\bf a}\in \mathbb{R}^d$ and $B$, which is still a $G$-Brownian motion, and $\sigma_{\bf a\ltrans{a}}:=\mathbb{E}[({\bf a}, B_1)^2]$.

\begin{definition}\label{hao}
For each $\eta\in M^0([0, T]; \mathbb{R})$, we define the It\^o type integral
$$
\mathcal{I}_{[0, T]}(\eta)=\int^T_0\eta_tdB^{\bf a}_t:=\sum^{N-1}_{k=0}\xi_k(B^{\bf a}_{t_{k+1}}-B^{\bf a}_{t_k}).
$$
Then, thanks to $G$-It\^o's inequality (cf. Lemma 3.4 in \cite{LIetPeng}), this linear mapping $\mathcal{I}_{[0, T]}$ on $M^0([0, T]; \mathbb{R})$ can be continuously extended to $\mathcal{I}_{[0, T]}: M^2_G([0, T];\mathbb{R})\rightarrow L^2_G(\Omega_T)$ and for each $\eta \in M^2_G([0, T]; \mathbb{R})$, we define $\int^T_0\eta_tdB^{\bf a}_t:=\mathcal{I}_{[0, T]}(\eta)$.
\end{definition}
Moreover, we have the following BDG type inequality (cf. Theorem 2.1 in \cite{Gao}). Define
$$\sigma_{\mathbf{a}\ltrans{\mathbf{a}}}:=\sup_{Q\in\Gamma}{\rm tr}(Q \ltrans{Q}\mathbf{a}\ltrans{\mathbf{a}}).$$
\begin{lemma}\label{th 2.1.1}Let $p\geq2$, $\mathbf{a}\in\mathbb{R}^d$, $\eta\in M_G^p([0,T]; \mathbb{R})$ and $0\leq s\leq t\leq T$. Then,
$$
\bar{\mathbb{E}}\bigg[\sup_{s\leq u\leq
t}\bigg|\int^u_s\eta_rdB^{\mathbf{a}}_r\bigg|^p\bigg]\leq C_p\sigma^{p/2}_{\mathbf{a}\ltrans{\mathbf{a}}} \mathbb{E}\left[\left(\int^t_s |\eta_u|^2du\right)^{p/2}\right],
$$
where $C_p>0$ is a constant independent of $\mathbf{a}, \
\eta$ and $\Gamma.$
\end{lemma}
In the $G$-expectation framework the quadratic variation process $\langle B^{\bf a}\rangle$ is no longer deterministic, which is
formulated by
$$
\langle B^\mathbf{a}\rangle_t:=\lim_{\mu(\pi^N_{[0, T]})\rightarrow 0}\sum^{N-1}_{k=0}(B^\mathbf{a}_{t^N_{k+1}}-B^\mathbf{a}_{t^N_k})^2=(B^\mathbf{a}_t)^2-2\int^t_0B^\mathbf{a}_sdB^\mathbf{a}_s,
$$
where $\pi^N_{[0, T]}$ is a partition of $[0,T]$, i.e., $\pi^N_{[0, T]}=\{t_0, t_1,\ldots,t_N\}$ such that
$0=t_0<t_1<\ldots<t_N=T$, and $\mu(\pi^N_{[0, T]}):=\max_{1\leq i\leq N}|t^N_i-t^N_{i-1}|$.
 For two given vectors $\mathbf{a}$, $\bar{\mathbf{a}}\in\mathbb{R}^d$, the mutual variation process of $B^\mathbf{a}$ and
$B^{\bar{\mathbf{a}}}$ is defined by $$\langle
B^\mathbf{a},B^{\bar{\mathbf{a}}} \rangle_t:=\frac{1}{4}(\langle
B^{\mathbf{a}+\bar{\mathbf{a}}}\rangle_t-\langle
B^{\mathbf{a}-\bar{\mathbf{a}}}\rangle_t).$$
By Corollary 5.7 in Chapter III of Peng \cite{Peng4},
for each $0\leq s\leq t\leq T,$
\begin{equation}\label{pquar}
\langle B^\mathbf{a}\rangle_t-\langle B^\mathbf{a}\rangle_s\leq \sigma_{\mathbf{a}\ltrans{\mathbf{a}}}(t-s).
\end{equation}
Let $B^i$ denote the $i$th coordinate of the $G$-Brownian motion $B$ and set $\langle B, B\rangle_t=(\langle B^i, B^j\rangle_t)_{i, j= 1, \ldots, d}$. Thus, the path of $\langle B, B\rangle$ quasi-surely has a bounded density and indeed the stochastic integral for $\eta\in M^1_G([0, T]; \mathbb{R})$  with respect to $\langle B^{\bf a}, B^{\bar{\bf a}}\rangle$ could be defined pathwisely.\\

Finally, we recall that  Gao \cite{Gao} proves the $G$-It\^{o} type integral $X_\cdot=\int_0^\cdot\eta_sdB^{\bf a}_s$ has a continuous $\bar{C}$-modification, for any $\eta\in M_G^2([0,T]; \mathbb{R})$.

\subsection{Conditions on the domain}
In order to investigate the reflected $G$-Brownian motion in this paper, we shall first recall the results in \cite{SaishoetTanaka}  for the deterministic Skorohod problem in a domain $D\subset\mathbb{R}^d$, $d\in \mathbb{N}^*$. In that paper,  the following conditions are assumed:\\

\textbf{CONDITION (A).} For $x\in \partial D$, we denote
$$\mathcal{N}_{x,r}=\{\mathbf{n}\in \mathbb{R}^d: |\mathbf{n}|=1, B(x-r\mathbf{n},r)\cap
D=\emptyset\},\ r>0\quad \mbox{and}\quad \mathcal{N}_{x}=%
\bigcup_{r>0}\mathcal{N}_{x,r},$$
where $B(z, r):=\{y\in \mathbb{R}^d: |y-z|<r|\}$, for $z\in \mathbb{R}^d$.
We assume that there exists a constant $r_{0}>0$ such that $\mathcal{N}_{x}=%
\mathcal{N}_{x,r_{0}}\neq \emptyset$, for all $x\in \partial D$.\\

\textbf{CONDITION (B).} Assume that there exist constants $\delta >0$ and $%
\beta \in \lbrack 1,\infty )$ such that for any $x\in \partial
D$, we can find a unit vector  $l_{x}$ such that
\begin{equation*}
\langle l_{x},\mathbf{n}\rangle \geq \frac{1}{\beta },\text{ \ \ for
all \ \ }\mathbf{n\in }\bigcup_{y\in B(x,\delta )\cap \partial D}\mathcal{N}%
_{y},
\end{equation*}%
where $\langle \cdot, \cdot\rangle$ denotes the usual inner product in $\mathbb{R}^d$.\\

Throughout this paper, we consider a domain
$D\subset \mathbb{R}^{d}$ satisfying both Condition (A) and (B). For each
$x\in \mathbb{R}^{d}$ such that $\mbox{dist}(x,\overline{D})<r_{0}$, there exists a unique
$\overline{x}\in \overline{D}$
with
$\left\vert x-\overline{x}\right\vert =\mbox{dist}(x,\overline{D})$. If $x\notin \overline{D}$
we have $\overline{x}\in \partial D$ and $\frac{\overline{x}-x}{\left\vert \overline{x%
}-x\right\vert }\in \mathcal{N}_{\overline{x}}$  (see e.g. Remark 1.3 in \cite{Saisho1}). We keep this notation $\overline{x}$ for the projection of $x$ on $D$ in the remainder of this paper.\\

To consider the solvability of reflected multi-dimensional Skorohod stochastic differential equations on the domain $D$, we assume furthermore \\

\textbf{CONDITION (C).} There exists a bounded function
$\Psi \in \mathcal{C}_{b}^{2}\left( \mathbb{R}%
^{d}\right) $ whose first and second derivatives are also bounded, and there exists a $\delta'>0$, such that
\begin{equation*}
\forall x\in \partial D,~\forall y\in \overline{D},~\forall \mathbf{n\in }%
\mathcal{N}_{x},~\left\langle y-x,\mathbf{n}\right\rangle +\frac{1}{\delta'}%
\left\langle \nabla \Psi \left( x\right) ,\mathbf{n}\right\rangle \left\vert
y-x\right\vert ^{2}\geq 0.
\end{equation*}
We note the bound of $\Psi$ and its derivatives by $L_\Psi$. This condition is critical for proving the a priori estimate (\ref{est10}), which is similar to its analogue in \S 3 of \cite{LionsetSznitman}.

\subsection{Deterministic Skorohod problem}
We recall here the solvability result of deterministic Skorohod problems in the domain $D$ satisfying Conditions $(A)$ and (B), which could be found  in \cite{Saisho1}. This result is our starting point of this paper. \\

Assume that $\phi$ is a continuous function taking values in $\mathbb{R}^{d}$ and that $\phi$ is of bounded variation over each finite interval. We denote by
$|\phi |_{t}$ the total variation of $\phi$ over $[0, t]$, i.e.,
\begin{equation*}
\left\vert \phi \right\vert _{t}:=\sup_{0=t_{1}<t_{2}<\cdots
<t_{n}=t, n\in \mathbb{N}} \sum_{k=1}^{n}|\phi _{t_{k}}-\phi
_{t_{k-1}}|.
\end{equation*}%
We also note
$\left\vert\phi \right\vert _{t}^{s}:=|\phi |_{t}-|\phi |_{s},$ $0\leq s\leq t$.\\

For a continuous function $w$ defined on $[0, T]$, $T>0$,  taking values in $\mathbb{R}^{d}$
with
$w(0)=0$ and for $x_0\in \overline{D}$,
we consider the Skorohod equation below:
\begin{equation}
\xi _t=x_0+w_t+\phi_t,~~~t\in \left[ 0,T\right] .  \label{1.3}
\end{equation}%
\begin{definition}
We call a couple of functions
$\left( \xi ,\phi
\right)$
solution of (\ref{1.3}), if it satisfies (\ref{1.3}) and the following conditions:
\begin{enumerate}[(i)]
\item The function $\xi $ is continuons and takes values in $\overline{D}$;
\item The function $\phi $ is continuous and takes values in $\mathbb{R}^{d}$ with $\phi (0)=0$. Moreover, it is of bounded variation over $[0, T]$ and for all $t\in \left[ 0,T\right]$,
\begin{eqnarray*}
\phi _t &=&\int_{0}^{t}\mathbf{n}_sd\left\vert \phi \right\vert _{s}; \\
|\phi |_{t} &=&\int_{0}^{t}{\bf 1}_{\{\xi _s\in\partial D\}}d\left\vert \phi
\right\vert _{s},
\end{eqnarray*}%
where $\mathbf{n}_s\in \mathcal{N}_{\xi _s}$, if $\xi _s\in \partial D$.
\end{enumerate}
\end{definition}
\begin{theorem}\label{det1}
(Theorem 4.1 in \cite{Saisho1}) Suppose that the domain $D\subset \mathbb{R}^d$ is open and satisfies Conditions (A) and (B). Then there exists a unique solution $\left( \xi ,\phi \right) $ for the deterministic Skorohod problem  (\ref{1.3}).
\end{theorem}

\section{Main results}
In this section, we present our main results on the reflected $G$-Brownian motion and on reflected stochastic differential equations driven by $G$-Brownian motion.
\subsection{Reflected $G$-Brownian motion}
We replace the deterministic function $w$ in the Skorohod problem (\ref{1.3}) by the $G$-Brownian motion $B$ and establish the following equation in the ``quasi-sure'' sense:
\begin{equation}
X_{t}=x_{0}+B_{t}+K_{t},~~x_0\in \overline{D},~~0\leq t\leq T  \label{1.4}.
\end{equation}%
\begin{definition}\label{defrgb}
We call a couple of processes  $\left( X,K\right) $ solves the Skorohod problem for the $G$-Brownian motion %
 (\ref{1.4}), if there exists a polar set $A$, such that
\begin{enumerate}[(i)]
\item The processes $X$ et $K$ belong to $M_{G }^{2}\left( \left[ 0,T%
\right]; \mathbb{R}^{d}\right) $, and for all $\omega\in A^c$,
\begin{equation*}
X_{t}(\omega)=x_{0}+B_{t}(\omega)+K_{t}(\omega),~~0\leq t\leq T;
\end{equation*}
\item For all $\omega\in A^c$, $X(\omega)$ is continuous and takes values in $\overline{D}$;
\item For all $\omega\in A^c$, $K(\omega)$ is continuous and takes values in $\mathbb{R}^{d}$ with $K_0(\omega)=0$. Moreover, $K(\omega)$ is of bounded variation over $[0, T]$ and for all $t\in \left[ 0,T\right]$,
\begin{eqnarray*}
K_{t}\left( \omega \right) &=&\int_{0}^{t}\mathbf{n}_s\left(\omega\right) d\left\vert K\right\vert _{s}\left( \omega \right);\\
\left\vert K\right\vert _{t}\left( \omega \right) &=&\int_{0}^{t}\mathbf{1}%
_{\left\{ X_{s}\left( \omega \right) \in \partial D\right\} }d\left\vert
K\right\vert _{s}\left( \omega \right) ,
\end{eqnarray*}
where $\mathbf{n}_s\left(\omega\right)\in \mathcal{N}_{X_{s}\left(
\omega \right)}$,  if $X_{s}\left(
\omega \right)\in \partial D$.
\end{enumerate}
In addition, we call $X$ reflected $G$-Brownian motion on the domain $D$.
\end{definition}
We have the following existence and uniqueness theorem for the reflected $G$-Brownian motion. The proof of this theorem is postponed to the next section.
\begin{theorem}\label{trgb}
Suppose that the domain $D\subset \mathbb{R}^d$ is open and satisfies Conditions (A) and (B). Then there exists a couple $\left( X ,K \right)\in \left(M^2_G([0, T]; \mathbb{R}^d)\times M^2_G([0, T]; \mathbb{R}^d)\right)$ which solves the Skorohod problem  (\ref{1.4}) whenever $x_0\in \overline{D}$. Moreover, if the problem (\ref{1.4}) admits two solutions $(X, K)$ and $(X', K')$, then the exists a polar set $\widetilde A$, such that for all $\omega\in \widetilde A^c$,
$$
X(\omega)=X'(\omega)~~~~and~~~K(\omega)=K'(\omega),\ \ 0\leq t\leq T.
$$
\end{theorem}
\subsection{Reflected stochastic differential equations driven by $G$-Brownian motion}
In addition to the reflected $G$-Brownian motion, we shall study reflected stochastic differential equations driven by $G$-Brownian motion, which is formulated as
\begin{align}
X_{t}=x_{0}+\int_{0}^{t}f\left(
s,X_{s}\right) ds&+\int_{0}^{t}h^{ij}\left( s,X_{s}\right) d\langle B^i, B^j\rangle_{s}\notag\\&+\int_{0}^{t}g^{j}\left( s,X_{s}\right) dB^j_{s}+K_{t},~~~0\leq t\leq T,~~~~\text{q.s.}.  \label{e1}
\end{align}
Here we adopt the Einstein summation convention.
In (\ref{e1}), the process $\langle B, B\rangle$ is the covariation matrix of the $d$-dimensional $G$-Brownian motion $B$. In what follows, we assume that the functions $f$, $h$, $g$ satisfy the following conditions:
\begin{assumption}
The functions $f$, $h^{i j}$, $g^{j}: \Omega\times [0, T]\times \overline{D}\longrightarrow \mathbb{R}^d$, $i$, $j=1, 2, \ldots d$, are functions such that \\
(\textbf{H1}) For all $x\in \overline{D}$, the processes $f(\cdot, x)$, $h^{i j}(\cdot, x)$, $g^{j}(\cdot, x)$ belong to
$M_G^{2}\left( \left[
0,T\right]; \mathbb{R}^d \right)$;\\
(\textbf{H2}) The functions $f$, $h^{i j}$, $g^{j}$ are uniformly bounded by $L_0$ and uniformly $L_0$-Lipschitz, i.e., there exists a constant $L_0>0$ such that for all $(\omega, t)\in \Omega\times [0, T]$,
\begin{equation*}
\|f\left( t,x\right) -f\left( t,y\right)\|+\left\|
h^{ij}\left( t,x\right) -h^{ij}\left( t,y\right) \right\| + \left\|
g^{j}\left( t,x\right) -g^{j}\left( t,y\right) \right\|\leq L_0\left\|
x-y\right\|,~~~~~\forall x,\ y\in \overline{D},
\end{equation*}
where $\|\cdot\|$ denotes the Hilbert-Schmidt norm for matrices.
\end{assumption}
\begin{definition}
We call a couple of processes  $\left( X,K\right) $ solves the Skorohod stochastic differential equation (\ref{e1}),
 if there exists a polar set $A$, such that
\begin{enumerate}[(i)]
\item The processes $X$ et $K$ belong to $M_G^{2}\left( \left[ 0,T%
\right]; \mathbb{R}^{d}\right) $ and satisfies (\ref{e1});
\item For all $\omega\in A^c$, $X(\omega)$  takes values in $\overline{D}$;
\item For all $\omega\in A^c$, $K(\omega)$ takes values in $\mathbb{R}^{d}$ with $K_0(\omega)=0$. Moreover, $K(\omega)$ is of bounded variation over $[0, T]$ and for all $t\in \left[ 0,T\right]$,
\begin{eqnarray*}
K_{t}\left( \omega \right) &=&\int_{0}^{t}\mathbf{n}_s\left(\omega\right) d\left\vert K\right\vert _{s}\left( \omega \right);\\
\left\vert K\right\vert _{t}\left( \omega \right) &=&\int_{0}^{t}\mathbf{1}%
_{\left\{ X_{s}\left( \omega \right) \in \partial D\right\} }d\left\vert
K\right\vert _{s}\left( \omega \right) ,
\end{eqnarray*}
where $\mathbf{n}_s\left(\omega\right)\in \mathcal{N}_{X_{s}\left(
\omega \right)}$,  if $X_{s}\left(
\omega \right)\in \partial D$.
\end{enumerate}
\end{definition}
Using a fixed point type argument, we shall prove in the next section the following existence and uniqueness theorem for the Skorohod stochastic differential equation (\ref{e1}).
\begin{theorem}\label{trgsde}
Suppose that the domain $D\subset \mathbb{R}^d$ is open and  satisfies Conditions (A) and (B). Then there exists a unique couple $\left( X ,K \right)\in \left(M^2_G([0, T]; \mathbb{R}^d)\times M^2_G([0, T]; \mathbb{R}^d)\right)$ which solves the Skorohod stochastic differential equation (\ref{e1}) whenever $x_0\in \overline{D}$ and the coefficients $f$, $h$, $g$ satisfy Assumptions (H1) and (H2).
\end{theorem}
\section{Proofs}
In this section, we shall prove Theorem \ref{trgb} and \ref{trgsde}. First, we recall the results for deterministic Skorohod problem in Saisho and Tanaka \cite{SaishoetTanaka} and provide an estimate when the function $w$ is $\alpha$-H\"older continuous, $\alpha\in (0, 1/2)$.
\subsection{Estimates for the deterministic Skorohod problem}
In (\ref{1.3}), we assume in addition that $w$ is an $\alpha$-H\"older continuous function on $[0, T]$, where $\alpha\in (0, 1/2)$, i.e.,
\begin{equation*}
\left\Vert w \right\Vert _{\alpha }=\sup_{0\leq s<t\leq T}\frac{%
\left\vert w _{t}-w _{s}\right\vert }{\left\vert t-s\right\vert
^{\alpha }}<\infty.
\end{equation*}%
Set
\begin{equation*}
\Delta _{s,t}(w ):=\sup \left\{\left\vert w_{t_{2}}-w
_{t_{1}}\right\vert :s\leq t_{1}<t_{2}\leq t\right\},
\end{equation*}%
and  $\Vert w\Vert_T:=\sup\{|w|_t: 0\leq t\leq T\}$.\\

We recall the penalization method in \cite{LionsetSznitman, SaishoetTanaka} and define a sequence of equations: for $m\in \mathbb{N}^*$,
\begin{equation}
\xi_{t}^{m}=x_{0}+w_{t}-\frac{m}{2}\int_{0}^{t}\nabla U(\xi_{s}^{m})ds,  \label{1.5}
\end{equation}
where $x_0\in \overline{D}$ and $U$ is a function satisfying (a) $U\in \mathcal C^{1}(\mathbb{R}^{d})$ and $U\geq 0$; (b) $U(x)=\left\vert x-\overline{x}\right\vert ^{2}$, if ${\rm dist}(x,\overline{D})\leq r_{0}/2$ ($r_{0}$ is from Condition (A)); (c) $\nabla U$ is bounded and Lipschitz. Indeed, the existence of such function $U$ is ensured by Condition (A) and (B). We denote by $2L$ the Lipschitz constant of $U$. The equation (\ref{1.5}) admits a unique solution $\xi^m$ which is continuous on $[0, T]$.\\

For each $m\in \mathbb{N}^*$, we define
\begin{equation}\label{eps}
\varepsilon
_{m}^{\alpha }\left( w \right) :=\frac{12e^{L}}{m^{\alpha }}\left\Vert
w \right\Vert _{\alpha }.
\end{equation}%
For $\varepsilon >0$, we note
\begin{equation*}
D_{\varepsilon }:=\left\{ x\in \mathbb{R}^{d},~{\rm dist}\left( x,\overline{D}%
\right) <\varepsilon \right\} .
\end{equation*}
Additionally, set
\begin{equation*}
\widetilde{\xi}_{t}^{m} :=\xi_{t/m}^{m},
\end{equation*}
it is easy to verify that
\begin{equation*}
\widetilde{\xi}_{t}^{m} =x_{0}+\widetilde{w}_{t/m} -\frac{1}{2}\int_{0}^{t}\nabla U(\widetilde{\xi}_{s}^{m} )ds,~~~0\leq t\leq mT.
\end{equation*}%

Similar to Lemma 4.2 in \cite{SaishoetTanaka}, we have the lemma below.
\begin{lemma}
\label{l4.2}  Assume that the domain $D\subset\mathbb{R}^d$ is open and satisfies Conditions (A) and (B). For each $m\in \mathbb{N}^{\ast }$ such that $0<\varepsilon _{m}^{\alpha }\left( w \right)
<r_{0}/2$ and for all $m'\geq m$, if there exists $u\in (0, m'T)$ such that
$\widetilde{\xi}_{u}^{m'} \in \partial D_{\varepsilon
_{m}^{\alpha }\left( w \right) /2}$, then $\{\widetilde{\xi}_{t}^{m'}\}_{u\leq t\leq T}$ hits  $\partial D_{\varepsilon
_{m}^{\alpha }\left( w\right) /3}$ before hitting $\partial
D_{\varepsilon _{m}^{\alpha }\left( w \right) }$.
\end{lemma}

\begin{proof}
This lemma can be proved by slightly modifying the proof of Lemma 4.2 in \cite{SaishoetTanaka}. Precisely, we could verify that for the given $\varepsilon
_{m}^{\alpha}$, the number $m$ itself is large enough to ensure Lemma 4.2 in \cite{SaishoetTanaka} if $\varepsilon_m^\alpha(w)$ is defined as (\ref{eps}). For the convenience of the reader, we briefly prove this lemma.\\

Fix $m'\geq m$ and suppose that $\widetilde{\xi}_{u}^{m'} \in \partial D_{\varepsilon
_{m}^{\alpha }\left( w \right) /2}$. Consider the auxiliary equation:
\begin{equation*}
\eta _t=\widetilde{\xi}_{u}^{m'} -\frac{1}{2}%
\int_{u}^{t}\nabla U(\eta _s)ds,~~~t\geq u.
\end{equation*}%
From Lemma 4.1 in \cite{SaishoetTanaka}, the above equation is solved by
\begin{equation*}
\eta _t=\overline{\widetilde{\xi}_{u}^{m'}}-\left(\overline{%
\widetilde{\xi}_{u}^{m'}}-\widetilde{\xi}_{u}^{m'}\right)\exp \{-(t-u)\},
\end{equation*}%
and the function $\eta$ satisfies
that for all $t\geq u$,
\begin{equation*}
\left\vert \overline{\eta _t}-\eta _t\right\vert =\frac{\varepsilon
_{m}^{\alpha }(w)}{2}\exp \{-(t-u)\}.
\end{equation*}%
Then, we denote
\begin{equation*}
u^{\prime }: =\inf \left\{t>u:\eta _t\in \partial D_{\varepsilon _{m}^{\alpha
}\left(w \right) /4}\right\},
\end{equation*}%
and  it is obvious that $u^{\prime }=u+\log 2<u+1$. \\

On the other hand, for $u\leq t\leq m'T$, we have
\begin{eqnarray}
\widetilde{\xi}_{t}^{m'} -\eta _t = w_{t/m'} -w_{u/m'}-\frac{1}{2}%
\int_{u}^{t}\left(\nabla U(\widetilde{\xi}_{s}^{m'} )-\nabla
U(\eta _s)\right) ds,\notag
\end{eqnarray}
which implies
\begin{eqnarray}
\left\vert \widetilde{\xi}_{t}^{m'}-\eta _t\right\vert
\leq \left\vert {w}_{t/m'} -w_{u/m'} \right\vert +L\int_{u}^{t}\left|\widetilde{\xi}_{s}^{m'} -\eta _s\right|ds.  \label{46}
\end{eqnarray}
\begin{itemize}
\item If $u^{\prime }\leq m'T$, then for all $u\leq t\leq u^{\prime
}$, we can deduce that
\begin{equation*}
 \left\vert {w}_{t/m'} -w_{u/m'} \right\vert \leq \frac{\left\Vert
w \right\Vert _{\alpha }}{{m'}^{\alpha }}\leq \frac{\left\Vert
w \right\Vert _{\alpha }}{{m}^{\alpha }}=\frac{\varepsilon
_{m}^{\alpha }\left( w \right)}{12e^L}.
\end{equation*}%
We apply Gronwall's lemma to (\ref{46}) and obtain
\begin{equation*}
\left\vert \widetilde{\xi}_{t}^{m'} -\eta _t\right\vert \leq \frac{\varepsilon
_{m}^{\alpha }\left( w \right)}{12e^L} e^{L(t-u)}<\frac{\varepsilon
_{m}^{\alpha }\left( w \right) }{12},~~~u\leq t\leq u^{\prime }.
\end{equation*}%
Therefore, for $u\leq t\leq u^{\prime }$,%
\begin{eqnarray*}
\left\vert \overline{\widetilde{\xi}_{t}^{m'}}-\widetilde{\xi}%
_{t}^{m'}\right\vert  \leq \left\vert \overline{\eta
_t}-\eta _t\right\vert +\left\vert \widetilde{\xi}_{t}^{m'}-\eta _t \right\vert
<\frac{\varepsilon _{m}^{\alpha }\left( w \right) }{2}+\frac{%
\varepsilon _{m}^{\alpha }\left( w \right) }{12}
<\varepsilon _{m}^{\alpha }\left( w \right),
\end{eqnarray*}%
whereas
\begin{eqnarray}\notag
\left\vert \overline{\widetilde{\xi}_{u^{\prime }}^{m'} }-%
\widetilde{\xi}_{u^{\prime }}^{m'}\right\vert  \leq
\left\vert \overline{\eta_{u^{\prime }}}-\eta_{u^{\prime }}\right\vert
+\left\vert \widetilde{\xi}_{u'}^{m'}-\eta_{u'} \right\vert
<\frac{\varepsilon _{m}^{\alpha }\left( w \right) }{4}+\frac{%
\varepsilon _{m}^{\alpha }\left( w \right) }{12}
=\frac{\varepsilon _{m}^{\alpha }\left( w \right) }{3},
\end{eqnarray}%
which implies that $\{\widetilde{\xi}_{t}^{m'}\}_{u\leq t\leq u'}$ hits $\partial D_{\varepsilon
_{m}^{\alpha }\left( w\right) /3}$ before hitting $\partial
D_{\varepsilon _{m}^{\alpha }\left( w \right) }$.
\item If $u^{\prime }>m'T$, then we could repeat the procedure above to prove that for $u\leq t\leq m'T$, $\left\vert \overline{\widetilde{\xi}_{t}^{m'}}-\widetilde{\xi}%
_{t}^{m'}\right\vert <\varepsilon _{m}^{\alpha }\left( w \right)$, which implies that $\{\widetilde{\xi}_{t}^{m'}\}_{u\leq t\leq m'T}$ never hits $\partial
D_{\varepsilon _{m}^{\alpha }\left( w \right) }$.
\end{itemize}
\end{proof}
We now give a proposition which is a straightforward corollary of Lemma \ref{l4.2}. The proof of this proposition is omitted and we refer the reader to Proposition 4.1 in \cite{SaishoetTanaka}.
\begin{proposition}
\label{p4.1}
Assume that the domain $D\subset\mathbb{R}^d$ is open and satisfies Conditions (A) and (B). For each $m\in \mathbb{N}^{\ast }$ such that $0<\varepsilon _{m}^{\alpha }\left( w \right)
<r_{0}/2$ and for all $m'\geq m$,
\begin{equation*}
\xi_{t}^{m'} \in D_{\varepsilon _{m}^{\alpha }\left(
w \right) },\quad\quad 0\leq t\leq T.
\end{equation*}
\end{proposition}
In particular, the results in the remainder of this subsection is based on the fact that if  $m\in \mathbb{N}^{\ast }$ is large enough such that $0<\varepsilon _{m}^{\alpha }\left( w \right)
<r_{0}/2$ then
\begin{equation*}
\xi_{t}^{m} \in D_{\varepsilon _{m}^{\alpha }\left(
w \right) },\quad\quad 0\leq t\leq T.
\end{equation*}

For $m$, $n\in \mathbb{N}^*$, set
\begin{eqnarray*}
T_{m,0} &=&\inf \left\{t\geq 0:\overline{\xi_{t}^{m} }\in
\partial D\right\};  \notag \\
t_{m,n} &=&\inf \left\{ t>T_{m,n-1}:\left\vert \overline{\xi_{t}^{m}}-\overline{\xi_{T_{m,n-1}}^{m} }%
\right\vert \geq \delta /2\right\}; \label{0.7} \\
T_{m,n} &=&\inf \left\{t\geq t_{m,n}:\overline{\xi_{t}^{m} }%
\in \partial D\right\},\notag
\end{eqnarray*}
where the constant $\delta$ is from Condition (B).
Moreover, denote by
$$
\phi^m_t:=-\frac{m}{2}\int_{0}^{t}\nabla U(\xi_{s}^{m})ds,~~0\leq
t\leq T.$$

In the remainder of this subsection, we shall provide an estimate of $|\phi^m|^0_T$. First, we prove that for sufficient large $m$, there is a lower bounded for $T_{m, n}-T_{m, n-1}$.\\

For simplicity, note $\gamma:=\frac{2\kappa^2(r_0/2)\beta}{r_0}$ and $\lambda(w):=\exp\left\{\gamma\left(\Vert w\Vert_T+\delta\right)\right\}$, where $\kappa(r_0/2)$ is a constant from Lemma 2.1 in \cite{SaishoetTanaka} such that
for all $|x-\overline{x}|<r_0/2$ and $|y-\overline{y}|<r_0/2$,  $|\overline{x}-\overline{y}|\leq \kappa(r_0/2)|x-y|$. Obviously, $\lambda> 1$.
\begin{lemma}
Assume that the domain $D\subset \mathbb{R}^d$ is open and  satisfies Conditions (A) and (B). For each $m\in \mathbb{N}^{\ast }$ such that $0<\varepsilon _{m}^{\alpha }\left( w \right)
<\frac{\delta}{180\beta\lambda(w)} \wedge r_{0}/2$ and for each $n\geq 1$ such that $T_{m, n}<\infty$,
$$
\left|T_{m, n}-T_{m, n-1}\right|\geq h:=\left(\frac{\delta}{36\beta\lambda(w)\Vert w\Vert_\alpha}\right)^{1/\alpha},
$$
\end{lemma}
\begin{proof} From Lemma 5.3 in \cite{SaishoetTanaka}, we have
\begin{align}
\Delta_{s, t}(\xi^m)\leq \left(8\beta\lambda(\omega)+1\right) \left(\Delta_{s, t}(w)+\varepsilon^\alpha_{m}(w)\right).\label{est98}
\end{align}
Then,
from Proposition \ref{p4.1} and the definitions of $t_{m, n}$ and $T_{m, n}$, we have
\begin{align*}
\frac{\delta }{2}-2\varepsilon _{m}^{\alpha }\left( w \right)&=\left\vert \overline{\xi_{t_{m,n}}^{m}%
 }-\overline{\xi_{T_{m,n-1}}^{m} }
 \right\vert -2\varepsilon _{m}^{\alpha }\left( w \right)\\
  &\leq \left\vert \xi_{t_{m,n}}^{m} -\xi_{T_{m,n-1}}^{m} \right\vert\\
   &\leq 9\beta\lambda(w) \left(\Delta_{T_{m,n}, T_{m, n-1}}(w)+\varepsilon^\alpha_{m}(w)\right),
\end{align*}%
which implies
$$
\Delta_{T_{m,n}, T_{m, n-1}}(w)\geq \frac{\delta-4\varepsilon^\alpha_m}{18\beta\lambda(w)}-\varepsilon^\alpha_m\geq \frac{\delta}{18\beta\lambda(w)}-5\varepsilon^\alpha_m\geq \frac{\delta}{36\beta\lambda(w)}.
$$
Furthermore, we deduce the desired result by the definition of $||w||_\alpha$.
\end{proof}
\begin{proposition}\label{prop44}
Assume that the domain $D\subset \mathbb{R}^d$ is open and satisfies Conditions (A) and (B). For each $m\in \mathbb{N}^{\ast }$ such that $0<\varepsilon _{m}^{\alpha }\left( w \right)
<\frac{\delta}{180\beta\lambda(w)} \wedge r_{0}/2$, we have
\begin{equation}
\left\vert \phi^m \right\vert^0 _{T} \leq C_0\left(\left\Vert w
\right\Vert _{\alpha }^{1+1/\alpha }+\left\Vert w
\right\Vert _{\alpha }\right)\exp \left\{ \gamma \left( 1+ 1/\alpha\right) \left\Vert w\right\Vert
_{T}\right\}  ,  \label{5.1}
\end{equation}%
where $C_0$ depends only on $\alpha$, $\beta$, $L$, $\delta$ and $T$.
\end{proposition}
\begin{proof}
For any $s$, $t$, such that $T_{m, n-1}\leq s\leq t\leq T_{m, n}$, we know from Lemma 5.1 in \cite{SaishoetTanaka} that
$$
\left|\phi^m\right|^s_t\leq \beta\left(\Delta_{s, t}(\xi^m)+\Delta_{s, t}(w)\right),
$$

We combine this inequality with (\ref{est98}) and deduce
$$
\left|\phi^m\right|^s_t\leq 10\beta^2\exp\left\{\gamma\left(\Vert w\Vert_T+\delta\right)\right\} \left(\Delta_{s, t}(w)+\varepsilon^\alpha_{m}(w)\right).
$$
Thus,
\begin{align*}
\left|\phi^m\right|^0_T
&\leq 10\left(\frac{T}{h}+1\right)\beta^2\exp\left\{\gamma\left(\Vert w\Vert_T+\delta\right)\right\} \left(\Delta_{0, T}(w)+\varepsilon^\alpha_{m}(w)\right)\\
&\leq 10\left(\frac{T}{h}+1\right)\beta^2\exp\left\{\gamma\left(\Vert w\Vert_T+\delta\right)\right\} \left(T^\alpha \Vert w\Vert_\alpha+\frac{12e^L}{m^\alpha}\Vert w\Vert_\alpha\right).
\end{align*}
By definition, $\exp\left\{\gamma\Vert w\Vert_T\right\} \geq 1$, then we can complete the proof by recalling the 
 the definition of $h$.
\end{proof}

\subsection{The existence and uniqueness for the reflected $G$-Brownian motion}
By Theorem \ref{det1}, for each $\omega\in \Omega$, there exists a pair $(X(\omega), K(\omega))$ that solves the deterministic Skorohod problem for $B(\omega)$, i.e.,
\begin{align}\label{eq32}
X(\omega)= x_0+ B(\omega)+K(\omega), \quad\quad x_0\in \overline{D}.
\end{align}
It is easy to see that the pair of processes $(X, K)$ satisfy (i), (ii) and (iii) in Definition \ref{defrgb}. Thus, to prove Theorem \ref{trgb},  it suffices to show that both $X$ and $K$ belong to $M_G^{2}\left( \left[ 0,T%
\right]; \mathbb{R}^{d}\right)$.
\begin{lemma}\label{lem45}
Assume that the domain $D\subset\mathbb{R}^d$ is open and satisfies Conditions (A) and (B). For each $\omega\in \Omega$, we define the pair of processes $(X, K)$ by the  unique solution of the deterministic Skorohod problem
$$
X(\omega)= x_0+ B(\omega)+K(\omega), \quad\quad x_0\in \overline{D},
$$
where $B$ is a $G$-Brownian motion. Then, $X$ and $K$ belong to $M_G^{2}\left( \left[ 0,T%
\right]; \mathbb{R}^{d}\right)$.
\end{lemma}
The main idea to prove this lemma is to construct a sequence of $\{(X^m, K^m)\}_{m\in \mathbb{N}^*}$ formed by elements from $(M^p_G([0, T]; \mathbb{R}^d)\times M^p_G([0, T]; \mathbb{R}^d))$ and then to show that the following convergences hold
$$
X^m\longrightarrow X, \quad\quad K^m\longrightarrow K, \quad\quad \mbox{in}\quad M^p_G([0, T]; \mathbb{R}^d),\ \ \ \ {\rm for}\ p\geq 2.
$$
Similarly to the previous subsection, we define for each $m\in \mathbb{N}^*$, 
\begin{equation*}
X_{t}^{m}=x_{0}+B_{t}-\frac{m}{2}\int_{0}^{t}\nabla U(X_{s}^{m})ds,~~0\leq
t\leq T,
\end{equation*}%
which is a stochastic differential equations driven by $G$-Brownian motion with bounded Lipschtiz coefficients. For each $m\in \mathbb{N}^*$, the above equation admits a unique solution in $M^p_G([0, T]; \mathbb{R}^d)$, $p\geq 2$ (cf. Theorem 4.2 in \cite{Gao}). Moreover, one can find a version of $X^m$, denoted still by $X^m$, such that there exists a polar set $A^m$, for all $\omega\in (A^m)^c$, $X^m$ is continuous and
$$
X_{t}^{m}(\omega)=x_{0}+B_{t}(\omega)-\frac{m}{2}\int_{0}^{t}\nabla U(X_{s}^{m}(\omega))ds,~~0\leq
t\leq T.
$$
We note $A:=\cup_{m\in \mathbb{N}^*} A^m$, which is still a polar set, and we note
\begin{equation*}
K_{t}^{m}\left( \omega \right) :=-\frac{m}{2}\int_{0}^{t}\nabla
U(X_{s}^{m}\left( \omega \right) )ds.
\end{equation*}

In what follows, we shall find a bound uniform in $m$ for $\{\mathbb{E}\left[\sup_{0\leq t\leq T} |X^m_t|^p\right]\}_{m\in \mathbb{N}^*}$ and $\{\mathbb{E}\left[(|K^m|^0_T)^p\right]\}_{m\in \mathbb{N}^*}$.\\

First, by Kolmogorov's Criterion (cf. Theorem 36 in Denis et al. \cite{DHP}), for any $\alpha\in (0, 1/2)$, there exists a polar set $A'$ such that for all $\omega\in (A')^c$, the path of $G$-Brownian motion $B(\omega)$ is $\alpha$-H\"older continuous.  Moreover, for any $p>0$, $\alpha\in (0, 1/2)$,
\begin{equation}\label{pmoment}
\mathbb{E}\left[\Vert B\Vert^p_\alpha\right]=\mathbb{E}\left[ \left( \sup_{0\leq s<t\leq T}\frac{\left\vert
B_{t}-B_{s}\right\vert}{\left\vert t-s\right\vert ^{\alpha }}\right)^p %
\right] <\infty.
\end{equation}
For each $m\in \mathbb{N}^*$, define
$$
\overline{A}^m:=\left(\left\{\omega\in\Omega: \varepsilon^\alpha_m(B_\cdot(\omega)) <\frac{\delta}{180\beta\lambda(B_\cdot(\omega))} \wedge \frac{r_{0}}{2}\right\}\cap A^c\cap (A')^c\right)^c.
$$

The following lemma gives an estimate for $c\left(\overline{A}^m\right)^c$.
\begin{lemma}
Fix $\alpha\in (0, 1/2)$. For $m\in \mathbb{N}^*$,
$$
c\left(\overline{A}^m\right)\leq \frac{C_{\alpha, p}}{m^p}, \quad\quad p\geq 1,
$$
where $C_{\alpha, p}$ depends on $T$, $\Gamma$, $\alpha$, $p$, $r_0$, $\delta$, $L$ and $\beta$.
\end{lemma}
\begin{proof} It is clear that
$$
 c\left(\overline{A}^m\right)\leq c\left(\left\{\omega\in \Omega: \varepsilon^\alpha_m(B_\cdot(\omega)) \geq  \frac{r_{0}}{2}\right\}\right)+ c\left(\left\{\omega\in \Omega: \varepsilon^\alpha_m(B_\cdot(\omega)) \geq \frac{\delta}{180\beta\lambda(B_\cdot(\omega))} \right\}\right).
$$
We calculate by Markov's inequality, for any $p\geq 1$,
\begin{align*}
c\left(\left\{\omega\in \Omega: \varepsilon^\alpha_m(B_\cdot(\omega)) \geq  \frac{r_{0}}{2}\right\}\right)&=c\left(\left\{\omega\in \Omega: \Vert B_\cdot(\omega))\Vert_\alpha\geq \frac{r_0e^{-L}}{24}m^\alpha\right\}\right)\leq \frac{C_1\mathbb{E}\left[\left\Vert B\right\Vert^{p/\alpha}_\alpha\right]}{m^p},
\end{align*}
where $C_1>0$ depends on $r_0$ and $L$. On the other hand,
\begin{align*}
&c\left(\left\{\omega\in \Omega: \varepsilon^\alpha_m(B_\cdot(\omega))\geq \frac{\delta}{180\beta\lambda(B_\cdot(\omega))} \right\}\right)\\
&\quad\quad\quad=c\left(\left\{\omega\in\Omega: \Vert B_\cdot (\omega)\Vert_\alpha \exp \left(\gamma \Vert B_\cdot (\omega)\Vert_T\right)\geq \frac{\delta e^{-L}}{2160\beta \exp\{\gamma\delta\}} m^\alpha\right\}\right)\\
&\quad\quad\quad\leq \frac{C_2\mathbb{E}\left[\Vert B\Vert_\alpha^{p/\alpha}\exp \left\{\frac{\gamma p}{\alpha}\Vert B\Vert_T\right\}\right]}{m^p}\\
&\quad\quad\quad\leq \frac{C_2\mathbb{E}\left[\Vert B \Vert_\alpha^{2p/\alpha}\right]^{1/2}\mathbb{E}\left[\exp \left\{\frac{2\gamma p}{\alpha}\Vert B\Vert_T\right\}\right]^{1/2}}{m^p},
\end{align*}
where $C_2>0$ depends on $r_0$, $\delta$, $L$ and $\beta$. By Theorem 3.3 in Luo and Wang \cite{LW}, we have
$$
\mathbb{E}\left[\exp \left\{\frac{2\gamma p}{\alpha}\Vert B\Vert_T\right\}\right]\leq C',
$$
where $C'$ depends on $T$, $\Gamma$, $p$, $\alpha$ and $\gamma$.
We combine this with (\ref{pmoment}) to conclude the desired result.
\end{proof}
\begin{proposition}\label{prop47}
Assume that the domain $D\subset \mathbb{R}^d$ is open and satisfies Conditions (A) and (B). Fix $\alpha\in (0, \frac{1}{2})$, then we have
$$
\mathbb{E}\left[ \sup_{0\leq t\leq T}\left\vert X_{t}^{m}\right\vert ^{p}%
\right] +\mathbb{E}\left[ \left( \left\vert K^{m}\right\vert^0 _{T}\right) ^{p}%
\right] \leq C' _{\alpha, p},
$$
where $C'_{\alpha, p}$ depends on $T$, $\Gamma$, $\alpha$, $p$, $r_0$, $\delta$, $L$ and $\beta$.
\end{proposition}
\begin{proof}
We denote by $\Lambda$ an upper bound of $\left\vert \nabla U\right\vert$, then
\begin{equation}\label{ka}
\mathbb{E}\left[ \left( \left\vert K^{m}\right\vert^0 _{T}\right) ^{p}\mathbf{1%
}_{\overline{A}^{m}}\right] \leq \left( \frac{m}{2}T\Lambda \right) ^{p}\overline{C}%
\left( \overline{A}^{m}\right) =C_{\alpha, p}\left(\frac{T\Lambda}{2}\right)^p.
\end{equation}
From (\ref{5.1}), we have, for each $\omega\in\left (\overline{A}^m\right)^c$,
\begin{equation}
\left\vert K^{m}\right\vert^0_{T}\left( \omega \right) \leq C_0 \left(\left\Vert B_{.}\left( \omega \right)
\right\Vert _{\alpha }^{1+1/\alpha }+\left\Vert B_{.}\left( \omega \right)
\right\Vert _{\alpha }\right)\exp \left\{ \gamma \left( 1+\frac{1}{%
\alpha }\right) \left\Vert B_{.}\left( \omega \right) \right\Vert
_{T}\right\},\label{kac}
\end{equation}%
where $C_0$ is the constant from (\ref{5.1}). Thus,
\begin{align*}
\mathbb{E}\left[ \left( \left\vert K^{m}\right\vert^0 _{T}\right) ^{p}\mathbf{1%
}_{(A^{m})^c}\right] \leq C_0\left(\mathbb{E}\left[\Vert B\Vert_\alpha^{2p(1+\alpha)/\alpha}\right]^{1/2}+\mathbb{E}\left[\Vert B\Vert_\alpha^{2p}\right]^{1/2}\right)\mathbb{E}\left[\exp \left\{\frac{2\gamma p(1+\alpha)}{\alpha}\Vert B\Vert_T\right\}\right]^{1/2}.
\end{align*}
Recall that for some $C_p>0$, which depends only on $p$,
\begin{equation*}
\left\vert X_{t}^{m}\right\vert ^{p}\leq C_{p}\left( \left\vert
x_{0}\right\vert ^{p}+\left\vert B_{t}\right\vert ^{p}+\left\vert
K_{t}^{m}\right\vert ^{p}\right),
\end{equation*}%
we can deduce that
\begin{equation*}
\sup_{0\leq t\leq T}\left\vert X_{t}^{m}\right\vert ^{p}\leq C_{p}\left(
\left\vert x_{0}\right\vert ^{p}+\sup_{0\leq t\leq T}\left\vert
B_{t}\right\vert ^{p}+\left( \left\vert K^{m}\right\vert^0_{T}\right)
^{p}\right).
\end{equation*}%
We take the $G$-expectation on both sides and apply the BDG type inequality, (\ref{ka}) and (\ref{kac}) to conclude the desired result.
\end{proof}
Now we are ready to prove Lemma \ref{lem45}.
\begin{proof}[Proof of Lemma \ref{lem45}]
From (6.4) in \cite{SaishoetTanaka}, we have for $\omega\in (A\cup A')^c$,
$$
X^m_t(\omega)\longrightarrow X_t(\omega), \ \ \ \ \mbox{uniform\ on}\ [0, T],
$$
and
for each $\omega\in (\overline{A}^m)^c$,
\begin{align*}
\sup_{0\leq s\leq t}|X^m_s(\omega)-X_s(\omega)|^2\leq 4\varepsilon^\alpha_m(\omega)&\left(\left|K^m(\omega)\right|^0_t
+\left|K(\omega)\right|^0_t\right)\\&+\frac{\gamma}{2\beta}
\int^t_0 \sup_{0\leq u\leq s}\left(|X^m_u(\omega)-X_u(\omega)|^2\right)d\left(\left|K^m(\omega)\right|^0_s+\left|K(\omega)\right|^0_s\right),
\end{align*}
which implies
\begin{align*}
\sup_{0\leq t\leq T}|X^m_t(\omega)-X_t(\omega)|^2\leq 4\varepsilon^\alpha_m(\omega)\left(\left|K^m(\omega)\right|^0_T
+\left|K(\omega)\right|^0_T\right)\exp\left\{\frac{\gamma}{2\beta}\left(\left|K^m(\omega)\right|^0_T+\left|K(\omega)\right|^0_T\right)\right\}.
\end{align*}
Now we shall prove that
\begin{equation}\label{import4}
\mathbb{E}\left[\sup_{0\leq t\leq T}|X^m_t(\omega)-X_t(\omega)|^2\right]\longrightarrow 0, \ \ \ \ \mbox{as}\ m\longrightarrow \infty.
\end{equation}
For $\epsilon>0$, by Markov's inequality, we could first fix a constant $M_0>0$, such that
\begin{equation}\label{import2}
\left(2C'_{\alpha, 4}\right)^{\frac{1}{2}}\mathbb{E}\left[{\bf 1}_{\Vert B_\cdot\Vert_\alpha\geq M_0}\right]\leq \frac{\left(2C'_{\alpha, 4}\right)^{\frac{1}{2}}\mathbb{E}\left[{\Vert B_\cdot\Vert_\alpha}\right]}{M_0}\leq \frac{\epsilon}{2},
\end{equation}
where $C'_{\alpha, 4}$ is the constant from Proposition \ref{prop47}. Then, we choose $m_0\in\mathbb{N}^*$ sufficiently large such that
\begin{align}\label{import1}
\frac{12e^L}{m_0^\alpha}M_0\leq \frac{r_0}{2}\ \ \mbox{and}\ \  \frac{12e^L}{m_0^\alpha}M_0\leq \frac{\delta}{180\beta\exp\left\{\gamma\left(\delta+ M_0T^\alpha\right)\right\}};
\end{align}
\begin{align}\label{import}
\frac{48e^L}{m_0^\alpha}M_0\left(4C_0M_0^{1+1/\alpha}\exp\left\{\gamma\left(1+\frac{1}{\alpha}\right)M_0T^\alpha\right\}\right)
\exp\left\{ \frac{2\gamma C_0}{\beta}M_0^{1+1/\alpha}\exp\left\{\gamma\left(1+\frac{1}{\alpha}\right)M_0T^\alpha\right\}\right\}\leq  \frac{\epsilon}{2},
\end{align}
where $C_0$ is the constant from Proposition \ref{prop44}.
From (\ref{import1}), we know for $\omega\in \{\Vert B_\cdot\Vert_\alpha< M_0\}\cap A^c\cap A'^c$ and $m\geq m_0$,
$$
\varepsilon^\alpha_m(B_\cdot(\omega)) <\frac{\delta}{180\beta\lambda(B_\cdot(\omega))} \wedge \frac{r_{0}}{2}.
$$
It follows that for $m\geq m_0$,
\begin{align*}
\mathbb{E}\left[ \sup_{0\leq t\leq T}\left\vert X_{t}^{m}-X_{t}\right\vert
^{2}\right] \leq \mathbb{E}&\left[ \sup_{0\leq t\leq T}\left\vert
X_{t}^{m}-X_{t}\right\vert ^{2}\mathbf{1}_{\{\Vert B_\cdot\Vert_\alpha< M_0\}\cap A^c\cap A'^c}\right] \\
&+\mathbb{E}\left[ \sup_{0\leq t\leq T}\left\vert X_{t}^{m}-X_{t}\right\vert
^{2}\mathbf{1}_{\{\Vert B_\cdot\Vert_\alpha\geq M_0\}\cup A\cup A'}\right] \\
&\leq \frac{\epsilon}{2}+\mathbb{E}\left[ \sup_{0\leq t\leq T}\left\vert X_{t}^{m}-X_{t}\right\vert
^{4}\right]^{\frac{1}{2}}\mathbb{E}\left[\mathbf{1}_{\{\Vert B_\cdot\Vert_\alpha\geq M_0\}}\right]\leq \epsilon,
\end{align*}
where the last inequality is deduced from (\ref{import2}) and (\ref{import}). Therefore, (\ref{import4}) holds true. From (\ref{import4}), it is obvious that 
$$
\mathbb{E}\left[\sup_{0\leq t\leq T}|K^m_t(\omega)-K_t(\omega)|^2\right]\longrightarrow 0, \ \ \ \ \mbox{as}\ m\longrightarrow \infty.
$$
 We end the proof.
\end{proof}
\begin{proof}[Proof of Theorem \ref{trgb}]
We define pathwisely a couple $(X, K)$ by the solution of the deterministic problem (\ref{eq32}). Then, we apply Lemma \ref{lem45} to prove that  $X$ and $K$ belong to
$M_G^{2}\left( \left[ 0,T%
\right]; \mathbb{R}^{d}\right)$. Therefore, $(X, K)$ is a couple satisfying Definition \ref{defrgb}. The uniqueness of the solution is inherited from the pathwise uniqueness.
\end{proof}
Instead of the $G$-Brownian motion, if we consider a $G$-It\^o process as
\begin{equation}\label{eq7}
Y_{t}=\int_{0}^{t}\alpha _{s}ds+\int_{0}^{t}\eta _{s}^{ij}d\left\langle B^{i},B^{j}\right\rangle _{s}+\int_{0}^{t}\beta^j_{s}
dB_{s}^{j},~~~~0\leq t\leq T,
\end{equation}%
where  $\alpha$, $\eta^{i j}$, $\beta^{j}: \Omega\times [0, T]\longrightarrow \mathbb{R}^d$, $i$, $j=1, 2, \ldots d$, are bounded functions in $M^2_G([0, T]; \mathbb{R}^d)$, then a similar result holds due to the fact that for any $p\geq 2$,
$$
\mathbb{E}\left[\sup_{0\leq t\leq T}\left|Y_t\right|^p\right]\leq C_p,
$$
which can be easily obtained by the BDG type inequality.
\begin{corollary}\label{ritop}
Suppose that the domain $D\subset\mathbb{R}^d$ is open and satisfies Conditions (A) and (B). Then there exists a couple $\left( X ,K \right)\in \left(M^2_G([0, T]; \mathbb{R}^d)\times M^2_G([0, T]; \mathbb{R}^d)\right)$ which solves the Skorohod problem
\begin{align}\label{e3}
X_{t}=x_{0}+Y_{t}+K_{t},~~0\leq t\leq T,
\end{align}
whenever $x_0\in \overline{D}$ and $Y$ is defined by (\ref{eq7}). Moreover, if the above problem admits two solutions $(X, K)$ and $(X', K')$, then the exists a polar set $A$, such that for all $\omega\in A^c$,
$$
X(\omega)=X'(\omega)~~~~and~~~K(\omega)=K'(\omega),\ \ 0\leq t\leq T.
$$
\end{corollary}
\begin{remark}
Indeed, thanks to Proposition \ref{prop47}, we could have a stronger convergence instead of (\ref{import4}), that is, for any $p\geq 2$,
$$
\mathbb{E}\left[\sup_{0\leq t\leq T}|X^m_t(\omega)-X_t(\omega)|^p\right]\longrightarrow 0, \ \ \ \ \mbox{as}\ m\longrightarrow \infty.
$$
Thus, the couple of solution $(X, K)$ in both Theorem \ref{trgb} and Corollary \ref{ritop} could be found in $\left(M^p_G([0, T]; \mathbb{R}^d)\times M^p_G([0, T]; \mathbb{R}^d)\right)$.
\end{remark}
\subsection{The existence and uniqueness for the RGSDE}
Without loss of generality, we consider in this subsection the following equation instead of (\ref{e1}),
\begin{align}
X_{t}=x_{0}+\int_{0}^{t}f\left(
s,X_{s}\right) ds+\int_{0}^{t}g\left( s,X_{s}\right) dB_{s}+K_{t},~~~0\leq t\leq T,~~~~\text{q.s.}.  \label{e2}
\end{align}
However, all result here holds for the more general case (\ref{e1}) due to the boundedness of the density of the process $\langle B, B\rangle$ (see \S III-4 in \cite{Peng4}). \\

If the coefficients $f$ and $g$ satisfy Assumptions (H1) and (H2) and (\ref{e2}) admits a solution couple $(X, K)$, then $(X, K)$ can be regarded as the solution couple of the Skorohod problem (\ref{e3}) in the domain $D$ for
$$
Y_t=\int_{0}^{t}f\left(
s,X_{s}\right) ds+\int_{0}^{t}g^{j}\left( s,X_{s}\right) dB^j_{s},~~~0\leq t\leq T.
$$
Then, it is straightforward that for any $p\geq 2$ there exists a constant $C'_p>0$ such that
$$
\mathbb{E}\left[ \sup_{0\leq t\leq T}\left\vert X_{t}\right\vert ^{p}%
\right] +\mathbb{E}\left[ \left( \left\vert K\right\vert^0 _{T}\right) ^{p}%
\right] \leq C' _{p}.
$$

\begin{proposition}\label{propest}
\label{p1} Suppose that the domain $D\subset \mathbb{R}^d$ is open and satisfies Conditions (A), (B) and (C). For $i=1, 2$, the couple $(\widetilde{X}^i, K^i)$ are solutions of the following Skorohod problems
\begin{align*}
\widetilde X_{t}^{i} =x_{0}+\int_{0}^{t}f^{i}\left( s,X_{s}^{i}\right) ds+\int_{0}^{t}g^{i}\left( s,X_{s}^{i}\right)
dB_{s}+K_{t}^{i},~~i=1,~2, \\
\left\vert K^{i}\right\vert _{t} =\int_{0}^{t}\mathbf{1}%
_{\left\{ \widetilde X^i_{s}\in \partial D\right\} }d\left\vert K^{i}\right\vert _{s}~%
{\rm and}\ K_{t}^{i}=\int_{0}^{t}{\bf n}_{s}^{i}d\left\vert K^{i}\right\vert
_{s}~{\rm with}\ {\bf n} _{s}^{i}\in {\mathcal{N}}_{{\widetilde X}_{s}^{i}},
\end{align*}%
where the coefficients $f^i$ and $g^i$ satisfy Assumptions (H1) and (H2).
Then, there exists a constant $C>0$ that depends on $\Gamma$, $d$, $\delta'$, $L_\Psi$ and $L_0$,
\begin{align}
\mathbb{E}\left[ \sup_{0\leq s\leq t}\left\vert
\widetilde X_{s}^{1}-\widetilde X_{s}^{2}\right\vert ^{4}\right] &+\mathbb{E}\left[ \sup_{0\leq
s\leq t}\left\vert K_{s}^{1}-K_{s}^{2}\right\vert ^{4}\right]\\
& \leq C\int^t_0\left(\mathbb{E}\left[ \sup_{0\leq u\leq s}\left\vert
 X_{u}^{1}- X_{u}^{2}\right\vert ^{4}\right]+
\mathbb{E}\left[ \sup_{0\leq u\leq s}\left\vert \hat{f}_{u}\right\vert ^{4}%
\right] +\mathbb{E}\left[ \sup_{0\leq u\leq s}\left\vert \hat{g}%
_{u}\right\vert ^{4}\right] \right) ds,\label{est10}
\end{align}%
where
$\hat{f}_{s}:=f^{1}\left( s,X_{s}^{2}\right)
-f^{2}\left( s,X_{s}^{2}\right) $ and $\hat{g}_{s}:=g^{1}\left( s,X_{s}^{2}\right) -g^{2}\left(
s,X_{s}^{2}\right)$.
\end{proposition}
\begin{proof} The proof is similar to the one of Lemma 3.1 in \cite{LionsetSznitman}, so we only display the key steps for the convenience of the readers. First, we have
\begin{eqnarray*}
\left\vert {\widetilde X}_{t}^{1}-{\widetilde X}_{t}^{2}\right\vert ^{2}
&=&2\int_{0}^{t}\left\langle {\widetilde X}_{s}^{1}-{\widetilde X}_{s}^{2},f^{1}\left(
s,X_{s}^{1}\right) -f^{1}\left( s,X_{s}^{2}\right) +\hat{f}_{s}\right\rangle
ds \\
&&+2\int_{0}^{t}\ltrans{\left({\widetilde X}_{s}^{1}-{\widetilde X}_{s}^{2}\right)}\left( g^{1}\left(
s,X_{s}^{1}\right) -g^{1}\left( s,X_{s}^{2}\right) +\hat{g}_{s}\right)
dB_{s} \\
&&+2\int_{0}^{t}\left\langle {\widetilde X}_{s}^{1}-{\widetilde X}_{s}^{2}, {\bf n}
_{s}^{1}\right\rangle d\left\vert K^{1}\right\vert _{s}
-2\int_{0}^{t}\left\langle {\widetilde X}_{s}^{1}-{\widetilde X}_{s}^{2},{\bf n}
_{s}^{2}\right\rangle d\left\vert K^{2}\right\vert _{s} \\
&&+\int_{0}^{t}\mbox{tr}\left[ \left( g^{1}\left( s,X_{s}^{1}\right)
-g^{1}\left( s,X_{s}^{2}\right) +\hat{g}_{s}\right) d\left\langle
B,B\right\rangle _{s}\ltrans{\left( g^{1}\left( s,X_{s}^{1}\right) -g^{1}\left(
s,X_{s}^{2}\right) +\hat{g}_{s}\right)} \right] ,
\end{eqnarray*}%
and
\begin{align*}
\Psi \left( {\widetilde X}_{t}^{i}\right) =\Psi \left( x_{0}\right)
&+\int_{0}^{t}\left\langle \nabla \Psi
\left( {\widetilde X}_{s}^{i}\right) ,f^{i}\left( s,X_{s}^{i}\right) \right\rangle ds+\int_{0}^{t} \ltrans{\left(\nabla \Psi \left( {\widetilde X}_{s}^{i}\right)\right)} g^{i}\left(
s,X_{s}^{i}\right) dB_{s}  \\
&+\int_{0}^{t}\left\langle \nabla \Psi \left( {\widetilde X}_{s}^{i}\right) ,{\bf n}
_{s}^{i}\right\rangle d\left\vert K^{i}\right\vert _{s} +\frac{1}{2}\int_{0}^{t}\mbox{tr}\left[{\bf H}(\Psi({\widetilde X}_s^i))
g^{i}\left( s,X_{s}^{i}\right) d\left\langle
B,B\right\rangle _{s}\ltrans{\left( g^{i}\left( s,X_{s}^{i}\right) \right)}%
\right].
\end{align*}%
Then,
\begin{align*}
&\exp\left\{-\frac{1}{{\delta'} }\left( \Psi \left( {\widetilde X}_{t}^{1}\right) +\Psi
\left( {\widetilde X}_{t}^{2}\right) \right) \right\} \times \left\vert
{\widetilde X}_{t}^{1}-{\widetilde X}_{t}^{2}\right\vert ^{2} \\
&\ \ \ \ =2\int_{0}^{t}\exp \left\{ -\frac{1}{{\delta'} }\left( \Psi \left(
{\widetilde X}_{s}^{1}\right) +\Psi \left( {\widetilde X}_{s}^{2}\right) \right) \right\} \\
&\ \ \ \ \ \ \ \ \ \ \ \ \times \bigg\{
\left\langle {\widetilde X}_{s}^{1}-{\widetilde X}_{s}^{2},  f^{1}\left(
s,X_{s}^{1}\right) -f^{1}\left( s,X_{s}^{2}\right) +\hat{f}_{s}\right\rangle ds\\
&\ \ \ \ \ \ \ \ \ \ \ \ \ \ \ \ +\ltrans{\left( {\widetilde X}_{s}^{1}-{\widetilde X}_{s}^{2}\right)} \left( g^{1}\left(
s,X_{s}^{1}\right) -g^{1}\left( s,X_{s}^{2}\right) +\hat{g}_{s}\right) dB_{s}\\
&\ \ \ \ \ \ \ \ \ \ \ \ \ \ \ \ +\left\langle {\widetilde X}_{s}^{1}-{\widetilde X}_{s}^{2}, {\bf n}
_{s}^{1}\right\rangle d\left\vert K^{1}\right\vert _{s}
-\left\langle {\widetilde X}_{s}^{1}-{\widetilde X}_{s}^{2},{\bf n}
_{s}^{2}\right\rangle d\left\vert K^{2}\right\vert _{s}\bigg\}
\\
&\ \ \ \ \ \ \ \ +\int_{0}^{t}\exp \left\{ -\frac{1}{{\delta'} }\left( \Psi \left(
{\widetilde X}_{s}^{1}\right) +\Psi \left( {\widetilde X}_{s}^{2}\right) \right) \right\} \\
&\ \ \ \ \ \ \ \ \ \ \ \ \times\mbox{tr}\left[\left( g^{1}\left( s,X_{s}^{1}\right) -g^{1}\left(
s,X_{s}^{2}\right) +\hat{g}_{s}\right) d\left\langle B,B\right\rangle
_{s}\ltrans{\left( g^{1}\left( s,X_{s}^{1}\right) -g^{1}\left( s,X_{s}^{2}\right) +%
\hat{g}_{s}\right)}\right] \\
&\ \ \ \ \ \ \ \  -\frac{1}{{\delta'} }\int_{0}^{t}\exp \left\{ -\frac{1}{{\delta'} }\left( \Psi
\left( {\widetilde X}_{s}^{1}\right) +\Psi \left( {\widetilde X}_{s}^{2}\right) \right) \right\}
\times \left\vert {\widetilde X}_{s}^{1}-{\widetilde X}_{s}^{2}\right\vert ^{2} \\
&\ \ \ \ \ \ \ \  \ \ \ \ \times\bigg\{
\left\langle \nabla \Psi \left({\widetilde X}_{s}^{1}\right), f^{1}\left(s,X_{s}^{1}\right)\right\rangle
+\left\langle \nabla \Psi \left( {\widetilde X}_{s}^{2}\right), f^{2}\left(s,X_{s}^{2} \right)\right\rangle ds\\
&\ \ \ \ \ \ \ \ \ \ \ \ \ \ \ \ +\left( \ltrans{\left(\nabla\Psi \left( {\widetilde X}_{s}^{1}\right)\right)} g^{1}\left(
s,X_{s}^{1}\right)+\ltrans{\left(\nabla\Psi \left({\widetilde X}_{s}^{2}\right)\right)} g^{2}\left(
s,X_{s}^{2}\right) \right) dB_{s} \\
&\ \ \ \ \ \ \ \ \ \ \ \ \ \ \ \ +\left\langle\nabla \Psi \left( \widetilde{X}_{s}^{1}\right), {\bf n} _{s}^{1}\right\rangle d\left\vert
K^{1}\right\vert _{s}
+\left\langle\nabla \Psi \left( \widetilde{X}_{s}^{2}\right), {\bf n} _{s}^{2}\right\rangle d\left\vert
K^{2}\right\vert _{s}\bigg\} \\
&\ \ \ \ \ \ \ \ -\frac{1}{2{\delta'} }\int_{0}^{t}\exp \left\{ -\frac{1}{{\delta'} }\left( \Psi
\left( {\widetilde X}_{s}^{1}\right) +\Psi \left( {\widetilde X}_{s}^{2}\right) \right) \right\}
\times \left\vert {\widetilde X}_{s}^{1}-{\widetilde X}_{s}^{2}\right\vert ^{2} \\
&\ \ \ \ \ \ \ \ \ \ \ \ \times\mbox{tr}\bigg[{\bf H}\left(\Psi\left( {\widetilde X}_{s}^{1}\right)\right)
g^{1}\left( s,X_{s}^{1}\right) d\left\langle B,B\right\rangle _{s}\ltrans{\left(
g^{1}\left( s,X_{s}^{1}\right) \right)}\\
&\ \ \ \ \ \ \ \ \ \ \ \ \ \ \ \ +{\bf H}\left(\Psi\left( {\widetilde X}_{s}^{2}\right)\right) g^{2}\left(
s,X_{s}^{2}\right) d\left\langle B,B\right\rangle _{s}\ltrans{\left( g^{2}\left(
s,X_{s}^{2}\right) \right)}\bigg] \\
&\ \ \ \ \ \ \ \ +\frac{1}{2}\left( \frac{1}{{\delta'} }\right) ^{2}\int_{0}^{t}\exp \left\{ -%
\frac{1}{{\delta'} }\left( \Psi \left( {\widetilde X}_{s}^{1}\right) +\Psi \left(
{\widetilde X}_{s}^{2}\right) \right) \right\} \left\vert {\widetilde X}_{s}^{1}-{\widetilde X}_{s}^{2}\right\vert
^{2} \\
&\ \ \ \ \ \ \ \ \ \ \ \ \times \left(\ltrans{\left(\nabla \Psi \left( {\widetilde X}_{s}^{1}\right)\right)} g^{1}\left(
s,X_{s}^{1}\right) +\ltrans{\left(\nabla \Psi \left( {\widetilde X}_{s}^{2}\right)\right)} g^{2}\left(
s,X_{s}^{2}\right)\right)\\
&\ \ \ \ \ \ \ \ \ \ \ \ \ \ \ \ \times d\left\langle B,B\right\rangle _{s}\left(\nabla \Psi \left(
{\widetilde X}_{s}^{1}\right) \ltrans{\left(g^{1}\left( s,X_{s}^{1}\right)\right)} +\nabla \Psi \left(
{\widetilde X}_{s}^{2}\right) \ltrans{\left(g^{2}\left( s,X_{s}^{2}\right) \right)}\right)\\
&\ \ \ \ \ \ \ \ -\frac{2}{{\delta'} }\int_{0}^{t}\exp \left\{ -\frac{1}{{\delta'} }\left( \Psi
\left( {\widetilde X}_{s}^{1}\right) +\Psi \left( {\widetilde X}_{s}^{2}\right) \right) \right\} \\
&\ \ \ \ \ \ \ \ \ \ \ \ \times\ltrans{\left( {\widetilde X}_{s}^{1}-{\widetilde X}_{s}^{2}\right)}\left( g^{1}\left(
s,X_{s}^{1}\right) -g^{1}\left( s,X_{s}^{2}\right) +\hat{g}_{s}\right)  \\
&\ \ \ \ \ \ \ \ \ \ \ \ \ \ \ \ \times d\left\langle B,B\right\rangle _{s}\left(\nabla \Psi \left(
{\widetilde X}_{s}^{1}\right) \ltrans{\left(g^{1}\left( s,X_{s}^{1}\right)\right)} +\nabla \Psi \left(
{\widetilde X}_{s}^{2}\right) \ltrans{\left(g^{2}\left( s,X_{s}^{2}\right) \right)}\right).
\end{align*}%
Thanks to Condition (C), we know that the integrals with respect to $d|K|$ are negative. Since the set $\Gamma$, the function $\Psi$ and its derivatives, the functions $f^i$ and $g^i$, $i=1, 2$, are bounded, we have
\begin{align*}
&\exp\left\{-\frac{2M}{{\delta'} } \right\} \times \left\vert
{\widetilde X}_{s}^{1}-{\widetilde X}_{s}^{2}\right\vert ^{2} \\
&\ \ \ \ \leq 2\int_{0}^{t}\exp \left\{ -\frac{1}{{\delta'} }\left( \Psi \left(
{\widetilde X}_{s}^{1}\right) +\Psi \left( {\widetilde X}_{s}^{2}\right) \right) \right\}\ltrans{\left( {\widetilde X}_{s}^{1}-{\widetilde X}_{s}^{2}\right)} \left( g^{1}\left(
s,X_{s}^{1}\right) -g^{1}\left( s,X_{s}^{2}\right) +\hat{g}_{s}\right) dB_{s}\\
&\ \ \ \ \ \ \ \ -\frac{1}{{\delta'} }\int_{0}^{t}\exp \left\{ -\frac{1}{{\delta'} }\left( \Psi
\left( {\widetilde X}_{s}^{1}\right) +\Psi \left( {\widetilde X}_{s}^{2}\right) \right) \right\}
\times \left\vert {\widetilde X}_{s}^{1}-{\widetilde X}_{s}^{2}\right\vert ^{2}\\
&\ \ \ \ \ \ \ \ \ \ \ \ \times\left( \ltrans{\left(\nabla\Psi \left( {\widetilde X}_{s}^{1}\right)\right)} g^{1}\left(
s,X_{s}^{1}\right)+\ltrans{\left(\nabla\Psi \left({\widetilde X}_{s}^{2}\right)\right)} g^{2}\left(
s,X_{s}^{2}\right) \right) dB_{s}\\
&\ \ \ \ \ \ \ \ +C\int^t_0  \left(\left\vert
{\widetilde X}_{s}^{1}-{\widetilde X}_{s}^{2}\right\vert ^{2}+ \left\vert
{X}_{s}^{1}-{X}_{s}^{2}\right\vert ^{2}+\left\vert
\hat{f}_{s}\right\vert^{2}+\left\vert
\hat{g}_{s}\right\vert^{2}\right)ds,
\end{align*}%
where $C>0$ is a constant that depends on $\Gamma$, $d$, $\delta'$, $L_\Psi$ and $L_0$, which may vary from line to line in the sequel.
We square both sides and apply the BDG type inequality to obtain
\begin{align*}
\mathbb{E}&\left[ \sup_{0\leq s\leq t}\left\vert
\widetilde X_{s}^{1}-\widetilde X_{s}^{2}\right\vert ^{4}\right]\\
& \leq C\int^t_0\left(\mathbb{E}\left[ \sup_{0\leq u\leq s}\left\vert
\widetilde X_{u}^{1}-\widetilde X_{u}^{2}\right\vert^{4}\right]+\mathbb{E}\left[ \sup_{0\leq u\leq s}\left\vert
 X_{u}^{1}- X_{u}^{2}\right\vert ^{4}\right]+
\mathbb{E}\left[ \sup_{0\leq u\leq s}\left\vert \hat{f}_{u}\right\vert ^{4}%
\right] +\mathbb{E}\left[ \sup_{0\leq u\leq s}\left\vert \hat{g}%
_{u}\right\vert ^{4}\right] \right) ds.
\end{align*}%
The desired result follows from the Gronwall inequality.
\end{proof}
\begin{proof}[Proof of Theorem \ref{trgsde}]
The uniqueness of solution is straightforward by Proposition \ref{propest}. We now turn to prove the existence. Indeed, by Corollary \ref{ritop} one can construct a sequence $\{(X^m, K^m)\}_{m\in \mathbb{N}^*}$ by the Picard type iteration starting with $X^0\equiv K^0=0$,
\begin{align*}
Y^{m+1}_t=\int_{0}^{t}f\left( s,X^m_{s}\right) ds+\int_{0}^{t}g\left( s,X^m_{s}\right)
dB_{s},\ \ \ \
X^{m+1}_{t} =x_{0}+Y^{m+1}_t+K^{m+1}_{t}, \\
\left\vert K^{m+1}\right\vert _{t} =\int_{0}^{t}\mathbf{1}%
_{\left\{ X^{m+1}_{s} \in \partial D\right\} }d\left\vert K^{m+1}\right\vert _{s}~%
{\rm and}\ K^{m+1}_{t}=\int_{0}^{t}{\bf n}^{m+1}_{s}d\left\vert K^{i}\right\vert
_{s}~{\rm with}\ {\bf n}^{m+1} _{s}\in {\mathcal{N}}_{{X^{m+1}_{s}}}.
\end{align*}
Thanks to the a priori estimate (\ref{est10}), we can proceed a similar argument as the proof of Theorem 2.1 and 4.1 in \cite{Gao} to find a couple of processes $(X, K)$ such that
\begin{align*}
\mathbb{E}\left[\sup_{0\leq t\leq T}|X^m_t-X_t|^2\right]+\mathbb{E}\left[\sup_{0\leq t\leq T}|K^m_t-K_t|^2\right]&\longrightarrow 0,\\
\mathbb{E}\left[\sup_{0\leq t\leq T}\left|Y^m_t-\int_{0}^{t}f\left( s,X_{s}\right) ds-\int_{0}^{t}g\left( s,X_{s}\right)
dB_{s}\right|^2\right]&\longrightarrow 0,
 \ \ \ \ \mbox{as}\ m\longrightarrow \infty,
\end{align*}
and such that there exists a polar set $A$ and subsequence $\{(X^{m_k}, K^{m_k})\}_{k\in \mathbb{N}^*}$, such that for each $\omega\in A^c$, $(X^{m_k}(\omega), K^{m_k}(\omega))$ is the solution couple for the deterministic Skorohod problem with $(Y^{m_k}(\omega), D)$, and
\begin{align*}
\sup_{0\leq t\leq T}|X^{m_k}_t(\omega)-X_t(\omega)|+\sup_{0\leq t\leq T}|K^{m_k}_t(\omega)-K_t(\omega)|&\longrightarrow 0, \\
\sup_{0\leq t\leq T}\left|Y^{m_k}(\omega)-\left(\int_{0}^{t}f\left( s,X_{s}\right) ds-\int_{0}^{t}g\left( s,X_{s}\right)
dB_{s}\right)(\omega)\right |&\longrightarrow 0,
\ \ \ \ \mbox{as}\ k\longrightarrow \infty.
\end{align*}
It is clear that $\left( X ,K \right)\in \left(M^2_G([0, T]; \mathbb{R}^d)\times M^2_G([0, T]; \mathbb{R}^d)\right)$. Besides, for each $\omega\in A^c$, $(X(\omega), K(\omega))$ verified (i) (ii) (iii) of Definiton  \ref{defrgb}, which can be proved by the last step of the proof to Theorem 4.1 in \cite{Saisho1}. We complete the proof.
\end{proof}

\bigskip
{\bf Acknowledgment:} Yiqing Lin gratefully acknowledges financial support from the European Research Council (ERC) under grant 321111. Abdoulaye SOUMANA HIMA is grateful for partial financial support from the Lebesgue Center of Mathematics (``Investissements d'avenir'' Program) under grant ANR-11-LABX-0020-01. We thank Ying HU for helpful suggestions.
\bigskip
\bibliography{mrgsde}
\bibliographystyle{plain}
\end{document}